\documentclass{article}

\usepackage{amsmath,amssymb,amsfonts,amsthm}

\def\newmathbf#1{\expandafter\newcommand\csname #1\endcsname{\mathbf{#1}}}
\def\newmathit#1{\expandafter\newcommand\csname #1\endcsname{\mathit{#1}}}
\def\newmathrm#1{\expandafter\newcommand\csname #1\endcsname{\mathrm{#1}}}
\def\newmathsf#1{\expandafter\newcommand\csname #1\endcsname{\mathsf{#1}}}
\def\newmathtt#1{\expandafter\newcommand\csname #1\endcsname{\mathtt{#1}}}

\def\newmathbb#1{\expandafter\newcommand\csname #1\endcsname{\mathbb{#1}}}
\def\newmathcal#1{\expandafter\newcommand\csname #1\endcsname{\mathcal{#1}}}
\def\newmathscr#1{\expandafter\newcommand\csname #1\endcsname{\mathscr{#1}}}
\def\newmathfrak#1{\expandafter\newcommand\csname #1\endcsname{\mathfrak{#1}}}

\def\newmathop#1{\expandafter\newcommand\csname #1\endcsname{\operatorname{#1}}}
\def\newmathoplim#1{\expandafter\newcommand\csname #1\endcsname{\operatorname*{#1}}}

\def\newmathbin#1{\expandafter\newcommand\csname #1\endcsname{\mathbin{\mathrm{#1}}}}
\def\newmathrel#1{\expandafter\newcommand\csname #1\endcsname{\mathrel{\mathrm{#1}}}}

\newcommand*\field[1]{\mathbb{#1}} 

\newcommand*\normal{\lhd}

\def\con#1=#2(#3){#1\equiv#2\pmod{#3}}

\def\myo{\ifmmode\circ\else$^\circ$\xspace\fi}  \let\o=\myo
\def\opri{\ifmmode^{\circ\prime}\else$^{\circ\prime}$\xspace\fi}

\newcommand*\Setst[2]%
        {\left\{\,#1\vphantom{#2} \;\right|\left. #2 \vphantom{#1}\,\right\}}

\newcommand*\gen[1]{\langle #1 \rangle}

\newcommand*\Genst[2]%
        {\left<\,#1\vphantom{#2} \;\right|\left. #2 \vphantom{#1}\,\right>}

\newcommand*\Z{{\field Z}}
\newcommand*\F{{\field F}}

\newmathop{rk}
\newmathop{rank}
\newmathop{pr}
\newmathop{m}

\theoremstyle{plain} 

\ifx\thesection\undefined
  \newtheorem{lemma}{Lemma}
\else
  \ifx\thelemma\undefined
    \newtheorem{lemma}{Lemma}[section]
  \else
  \fi
\fi
\newtheorem{theorem}[lemma]{Theorem}

\newtheorem{fact}[lemma]{Fact}

\newtheorem*{rawnamedtheorem}{\therawnamedtheorem}
\newcommand{\therawnamedtheorem}{\error}
\newenvironment{namedtheorem}[1]{\renewcommand{\therawnamedtheorem}{#1}
   \begin{rawnamedtheorem}}
  {\end{rawnamedtheorem}}

\newtheorem{Claim}[lemma]{Claim}

\newenvironment{verification}[1][Verification]%
   {\begin{proof}[#1]}{\end{proof}}

\makeatother%
\def\unexpandableprotect{\@unexpandable@protect}
\makeatletter%

\def\writeout#1#2{%
\begingroup%
\let\\=\relax%
\let\protect=\unexpandableprotect%
\immediate\write#1{#2}%
\endgroup%
}

\def\leqs{\leqslant}
\def\geqs{\geqslant}

\title{Uniqueness Cases in Odd Type Groups of Finite Morley Rank, Revisited}
\author{Alexandre V.\ Borovik%
\\
School of Mathematics, The University of Manchester\\
Oxford Road, Manchester M13 9PL, England\\
{\tt alexandre.borovik >>>at<<< gmail.com}
\and
Jeffrey Burdges\\
School of Mathematics and Statistics\\
University of St Andrews\\
Mathematical Institute\\
North Haugh, St Andrews KY16 9SS Scotland\\
{\tt burdges >>>at<<< gmail.com}
\and
Ali Nesin\\
Mathematics Department,  Istanbul Bilgi University\\
Ku\c{s}tepe \c{S}i\c{s}li, Istanbul, Turkey\\
{\tt anesin >>>at<<< bilgi.edu.tr}
}
\date{25 November 2011}

\begin{document}

\maketitle

\begin{abstract}
The paper contains versions of the Strong Embedding Theorem  and the Uniqueness Subgroup Theorem for groups of finite Morley rank and odd type which are needed for the study of permutations actions and modules in the finite Morley rank category.

\footnotetext{Mathematics Classification Numbers 03C60, 20G99}
\end{abstract}

\section{Introduction}\label{sec:Intro}

This paper relates to the programme of study of linear and permutation actions of finite Morley rank started by Borovik and Cherlin in \cite{borche} and continued in \cite{deloro}. The paper contributes important general structural results, the Strong Embedding Theorem  and the Uniqueness Subgroup Theorem stated below, to the forthcoming paper by Berkman and Borovik \cite{ABAB} about the pseudoreflection actions of groups of finite Morley rank.

\begin{namedtheorem}{Strong Embedding Theorem}
Let $G$ be a simple group of finite Morley rank and odd type with
Pr\"{u}fer\/ $2$-rank $\geqslant 3$.  Let\/ $S$ be a Sylow
$2$-subgroup of $G$ and $T=S^\circ$ the maximal $2$-torus in $S$.
Assume in addition that every proper definable subgroup of\/ $G$ containing $T$ is a $K$-group.

Then if\/ $G$ has a proper $2$-generated core
$M = \Gamma_{S,2}(G) < G$ then $M$ is strongly embedded.

\end{namedtheorem}

The Strong Embedding Theorem is a variation on the theme of an earlier version of a similar result under the same name \cite{bbn}.

\begin{namedtheorem}{Uniqueness Subgroup Theorem} Let $G$ be a simple group of finite Morley rank and odd type with
Pr\"{u}fer $2$-rank $\geqslant 3$.  Let $S$ be a Sylow
$2$-subgroup of\/ $G$ and $D=S^\circ$ the maximal\/ $2$-torus in $S$. Let further $E$ be the subgroup of\/ $D$ generated by involutions and $\theta$ a nilpotent signaliser functor on $G$.

Assume in addition that every proper definable subgroup of\/ $G$ containing $T$ is a $K$-group.

Then $M = N_G(\theta(E))$ is a  strongly embedded subgroup in $G$.

\end{namedtheorem}

An early version of the Uniqueness Subgroup Theorem was published as Theorem 6.7 in \cite{Bo95}.

The notions of both $2$-generated core and strongly embedded subgroup
arise as so-called {\em uniqueness cases} in finite group theory.
These subgroups both exhibit a black hole property reminiscent of a normal
subgroup; strong embedding, however, is far more powerful and has global consequences which deeply affect the structure of the group.

The standard references for our basic facts and terminology are \cite{BN} and \cite{abc}.
Here we give definitions immediately involved in the statement of the Strong Embedding Theorem.

A group $G$ of finite Morley rank is {\em connected} if it contains no
proper definable subgroup of finite index; the connected component $G^\circ$ of $G$ is the maximal connected subgroup of $G$. The  definable closure $d(X)$ of a subset $X \subseteq G$ is a minimal definable subgroup containing $X$. If $H \leqslant G$ is a na arbitrary (not necessarily definable) subgroup, we set $H^\circ = H \cap d(H)^\circ$; this is a subgroup of finite index in $H$.

We define the {\em $2$-rank} $m_2(G)$ of a group $G$ to be the maximum
rank of its elementary abelian 2-subgroups.  Also, the {\em Pr\"{u}fer $2$-rank}
$\pr_2(G)$ is the maximum 2-rank of its Pr\"{u}fer $2$-subgroups $\Z_{2^\infty}^k$.
These ranks must all be finite for subgroups of an odd type group of finite Morley rank.

Let $G$ be a group of finite Morley rank and let $S$ be a Sylow
2-subgroup of $G$.
We define the {\em $2$-generated core} $\Gamma_{S,2}(G)$ of $G$ to be
the definable hull of the group generated by all normalizers $N_G(U)$
of all elementary abelian 2-subgroups $U \leqs S$ with $m_2(U) \geqs 2$.

A proper definable subgroup $M$ of a group $G$ of finite Morley rank
is said to be {\em strongly embedded} if $M$ contains an involution but subgroups
$M \cap M^g$  do not contain involutions for any $g\in G \backslash M$.

Our next group of definition concerns signaliser functors.
For any involution $s \in G$, let $\theta(s) \leqslant  O(C_G(s))$ be some
connected definable  normal subgroup of $C_G(s)$. We say that $\theta$
is a \emph{signalizer functor} if, for any commuting involutions $t,s \in G$,
$$ \theta(t) \cap C_G(s) = \theta(s) \cap C_G(t). $$ The signalizer
functor $\theta$ is {\em complete\/}   if for any elementary abelian subgroup $E \leqslant  G$ of order at least
$8$ the subgroup $$ \theta(E) = \langle \theta(t) \mid t \in E^\#
\rangle $$ is a connected subgroup without $2$-torsion and
 $$
C_{\theta(E)}(s) = \theta(s) $$
for any $s \in E^\#$.
Finally, a signalizer functor $\theta$ is {\em non-trivial\/}  if
$\theta(s) \ne 1$ for some involution $s \in G$, and {\em nilpotent\/}
if all the subgroups $\theta(t)$ are nilpotent.

\section{Sylow Theory}\label{sec:Sylow}

\begin{fact}[\cite{BP}; Lemma 10.8 of \cite{BN}]
Let $S$ be a Sylow $2$-subgroup of a group of finite Morley rank.  Then
$S^\o = B * T$ is a central product of a definable connected nilpotent
subgroup $B$ of bounded exponent and of a $2$-torus $T$, i.e.~$T$ is a
divisible abelian $2$-group.
\end{fact}

A group is said to have {\em odd type} if $B = 1$ and $T \neq 1$.
This notion is well-defined because groups of finite Morley rank
have a good theory of Sylow 2-subgroups:

\begin{fact}[Theorem 10.11 of \cite{BN}]\label{Sylow}
The Sylow $2$-subgroups of a group of finite Morley rank are conjugate.
\end{fact}

We will find the following  corollary, known as a ``Frattini
argument'', to be useful.

\begin{fact}[Corollary 10.12 of \cite{BN}]\label{Sylow_frattini}
Let $G$ be a group of finite Morley rank, let $N \normal G$ be a
definable subgroup, and let $S$ be a characteristic subgroup of the
Sylow 2-subgroup of $N$.  Then $G = N_G(S) N$.
\end{fact}

The key result describing conjugation patterns of $2$-elements in a group of finite Morley rank is the following theorem by Burdges and Cherlin.

\begin{fact}[Theorem 3 of \cite{bc}]\label{torsion}
Let $G$ be a connected group of finite Morley rank,
and $a$ any $2$-element of $G$ such that $C_G(a)$ is of odd type. Then $a$ belongs to a $2$-torus.
\end{fact}

\begin{fact}[Lemma 4.5 of \cite{Bo95}] If $S$ is a $2$-subgroup in a group of finite Morley rank and odd type
with $\pr_2(S) \geqslant  3$, then
for any two four-subgroups\/ {\rm (}i.e.\ elementary abelian
subgroups of order {\rm 4}{\rm )} $U, V \leqslant  S$ there is a sequence of
four-subgroups
$$
U = V_1, V_2, \dots, V_n = V,
$$
such that $[V_i, V_{i+1}] =1$ for all $i = 1,2, \dots, n-1$.
 \label{lm:2-connected}
\end{fact}

If conclusions of fact~\ref{lm:2-connected} hold in a $2$-group $S$, we say that $S$ is \emph{$2$-connected}.

\section{The Generation Principle for $K$-groups}\label{sec:Kgrp}

A group $G$ of finite Morley rank is called a {\em $K$-group} if every non-trivial
connected definable simple section of $G$ is a Chevalley group over an
algebraically closed field. In particular, nilpotent and solvable groups of finite Morley rank are $K$-groups.

A {\em Klein four-group}, or just {\em four-group} for short, is a
group isomorphic to $\Z/2\Z \oplus \Z/2\Z$.  We will use the notation
$H^\# = H \setminus \{1\}$ to denote the set of non-identity elements
of a group $H$.

We need the following generation principle for $K$-groups:

\begin{fact}[Theorem 5.14 of \cite{Bo95}]\label{Kgrp_gen}
Let $G$ be a connected $K$-group of finite Morley rank and odd type.
Let $V$ be a four-subgroup acting definably on $G$.  Then
$$ G = \langle C^\o_G(v) \mid v\in V^\# \rangle $$
\end{fact}

Our next lemma is an easy consequence of Fact~\ref{Kgrp_gen}.

\begin{lemma}\label{stremb_CA}
Let $G$ be a simple group of finite Morley rank and odd type; assume, in addition, that centralisers of involutions in $G$ are $K$-groups.
Let $S$ be a Sylow $2$-subgroup of $G$ and let $M = \Gamma_{S,2}(G)$
be the $2$-generated core associated with $S$.
Let $A$ be an elementary abelian $2$-subgroup of $M$ with $m_2(A) \geqs 3$.
Then $C^\o_G(a) \leqs M$ for any $a\in A^\#$.
\end{lemma}

\begin{proof}
Let $K= C_G(a)$, then $A \leqslant K$ and $K$ is a $K$-group of odd type. Let $A_1$ be a four-subgroup of $A$ disjoint from $\gen{a}$.
By Fact~\ref{Kgrp_gen}, $K^\o = \gen{C_K^\o(x)\mid x\in A_1^\#}$.
Now $C_K^\o(x) \leqs C_G(a,x)$ and $\gen{a, x}$ is a four-subgroup
of $S$.  Thus $K^\o \leqs M$.
\end{proof}

\section{Strong Embedding}\label{sec:StrEmbDef}

$I(H)$ to denote the set of involutions of $H$.

We will apply the usual criteria for strong embedding:

\begin{fact}[Theorem 10.19 of \cite{BN}]\label{strembdef} 
Let $G$ be a group of finite Morley rank with a proper definable
subgroup $M$.  Then the following are equivalent:
\begin{enumerate}
\item $M$ is a strongly embedded subgroup.
\item $I(M) \neq \emptyset$, $C_G(i) \leqs M$ for every $i \in I(M)$,
      and $N_G(S) \leqs M$ for every Sylow $2$-subgroup of $M$.
\item $I(M) \neq \emptyset$ and $N_G(S) \leqs M$ for every non-trivial
      $2$-subgroup $S$ of $M$.
\end{enumerate}
\end{fact}

\section{Proof of the Strong Embedding Theorem}\label{sec:StrEmb}

Let $G$ be a simple group of finite Morley rank and odd type with
 Pr\"{u}fer $2$-rank $\geqslant 3$.  Let $S$ be a Sylow
$2$-subgroup of $G$ and $T = S^\circ$ a maximal $2$-torus in $G$.  Suppose that $G$ has a proper $2$-generated core
$M = \Gamma_{S,2}(G) < G$. Assume in addition that every proper definable subgroup of $G$ which contains $T$ is a $K$-group.

\begin{Claim}\label{2cK}
Let $t$ be a $2$-element in $G$. Then $t$ belongs to a $2$-torus and  $C_G(t)$ is a $K$-group.
\end{Claim}
\begin{verification} By Fact~\ref{torsion}, $t$ belongs to a maximal $2$-torus $R$ of $G$. Since by Sylow's Theorem $R^g = T$ for some $g\in G$, we see that $T\geqslant C_G(t)^g$ and therefore $C_G(t)^g$ and hence $C_G(t)$ is a $K$-group.
\end{verification}

Denote by $E$ the subgroup of $T$ generated by all involutions in $T$; then $E \normal S$ and $E$ is an elementary abelian 2-group with $m_2(E)\geq3$.

\begin{Claim}\label{stremb_E1}
For every $i\in I(S)$, $C_E(i)$ contains a four-group.
\end{Claim}

\begin{verification}
Since $E$ is normal in $S$, the involution $i$ induces a linear transformation
of the $\F_2$-vector space $E$.  Since $m_2(E) > 2$, the Jordan canonical form
of $i$ cannot consist of a single block, so there are at least two eigenvectors.
Since the eigenvalues associated to these eigenvectors must have order 2, the
eigenvalues must both be 1, as desired.
\end{verification}

\begin{Claim}\label{stremb_CconM}
$C^\o_G(i) \leq M$ for every $i\in I(M)$.
\end{Claim}

\begin{verification}
We may assume that $i\in I(S)$ by Fact~\ref{Sylow}.
By Claim~\ref{stremb_E1}, there is a four-group $E_1 \leq E$ centralized
by $i$.  Thus either $E$ or $\langle E_1,i \rangle$ is an elementary
abelian $2$-group of rank at least three which contains $i$.
By Fact~\ref{stremb_CA}, $C^\o_G(i) \leq M$.
\end{verification}

\begin{Claim}\label{TinM}
Let $R$ be a maximal $2$-torus in $M$. Then $N_G(R)\leqslant M$.
\end{Claim}

\begin{verification} Observe that
$N_G(R)\leqslant N_G(E) \leqslant M$ by definition of the proper $2$-generated core.
\end{verification}

\begin{Claim}\label{almost}
$C_G(i) \leqslant M$ for every involution $i \in M$.
\end{Claim}

\begin{verification}
By Claim~\ref{2cK}, $C_G(i)$ contains a maximal $2$-torus $R$ of $G$. Obviously, $R \leqslant C^\circ_G(i) \leqslant M$ by Claim~\ref{stremb_CconM}. Now by Claim~\ref{TinM} $N_G(R) \leqslant M$. Applying the Frattini argument to $C_G(i)$, we see that
\[
C_G(i) = C^\circ_G(i)N_{C_G(I)}(R) \subseteq C^\circ_G(i)N_G(R) \leqslant M\cdot M = M.\]
\end{verification}

\begin{Claim}\label{last}
$N_G(S) \leqslant M$.
\end{Claim}

\begin{verification} This follows from Claim~\ref{TinM} since
$N_G(S) \leqslant N_G(T) \leqslant M$.
\end{verification}

Claims~\ref{almost} and \ref{last} show that $M$ satisfies the criterion for being a strongly embedded subgroups (Fact~\ref{strembdef}) which completes the proof of the theorem.

\section{Proof of the Uniqueness Subgroup Theorem}

We start the proof by quoting a theorem on completeness of nilpotent signaliser functors.

\begin{theorem} \textbf{\emph{\cite[Theorem~B.30]{BN}}} Any
nilpotent signaliser functor $\theta$ on a group $G$ of finite Morley
rank is complete. \label{th:signfunct}
\end{theorem}

Now let $G$ be a group  of odd type of Pr\"ufer 2-rank at least $3$ and $S$ a Sylow $2$-subgroup in
$G$. Let $T$ be a maximal $2$-torus in $S$ and $E$ the subgroup generated by involutions in $T$; obviously, $E$ is
 a normal in $S$ elementary
abelian subgroup of order  at least 8. Let $\theta$ be a non-trivial nilpotent signaliser functor on $G$.
If $s \in S$ is
any involution and $\theta(s) \ne 1$, then let us take $D = C_E(s)$;
then $|D| \geqslant  4$.  Obviously $D$ normalizes $\theta(s)$ and $$
\theta(s) = \langle \theta(s) \cap C_t \,|\, t \in D^{\#} \rangle  \leqslant
\theta(E) $$ by Theorem~\ref{Kgrp_gen}.
So, if we assume that $\theta(s) \ne 1$ for some involution $s \in
P$, then we have $\theta(E)  \ne 1$.

Now we can state our main result about signaliser $2$-functors.

\begin{theorem}~~In the notation above, if a group $G$ of finite Morley rank and odd
type  has Pr\"{u}fer $2$-rank at least
$3$ and admits  a non-trivial nilpotent signaliser functor $\theta$,
then
$$
\Gamma_{S,2}(G) \leqslant  N_G(\theta(E)).
$$
In particular,  $G$ has a proper $2$-generated core.
 \label{th:sign2core}
\end{theorem}

Now application of the Strongly Embedding Theorem immediately yields the main result of this paper:

\begin{namedtheorem}{Uniqueness Subgroup Theorem} Let $G$ be a simple group of finite Morley rank and odd type with
Pr\"{u}fer $2$-rank $\geqslant 3$.  Let $S$ be a Sylow
$2$-subgroup of\/ $G$ and $D=S^\circ$ the maximal\/ $2$-torus in $S$. Let further $E$ be the subgroup of\/ $D$ generated by involutions and $\theta$ a nilpotent signaliser functor on $G$.

Assume in addition that every proper definable subgroup of\/ $G$ containing $T$ is a $K$-group.

Then $M = N_G(\theta(E))$ is a  strongly embedded subgroup in $G$.

\end{namedtheorem}

\paragraph{Proof of Theorem~{\rm \ref{th:sign2core}}.}
For an arbitrary  four-subgroup $W < S$ write
$$
\theta(W) = \langle \theta(w) \,|\, w \in W^{\#} \rangle.
$$
Notice that, if $V$ and $W$ are two commuting four-subgroups in $S$,
$[V,W]=1$, and $v \in V^{\#}$, then
$\theta(v)$ is $W$-invariant and, by Theorem~\ref{Kgrp_gen},
\begin{eqnarray*}
\theta(v) & = & \langle \theta(v) \cap C_w \,|\, w \in W^{\#} \rangle \\[.5ex]
& \leqslant  & \langle \theta(w) \,|\, w \in W^{\#} \rangle \\[.5ex]
&=& \theta(W).
\end{eqnarray*}
Therefore $\theta(V) \leqslant  \theta(W)$. Analogously $\theta(W) \leqslant
\theta(V)$ and $\theta(V) = \theta(W)$.

Recall that the Sylow
subgroup $S$ is $2$-connected by Lemma~\ref{lm:2-connected}; this means that if
$U$ is a four-subgroup from $E$ and $V$ is an arbitrary
four-subgroup in $S$ then there is a sequence of four-subgroups
$$
U=V_0, V_1, \ldots, V_n = V
$$
such that $[V_i, V_{i+1}] =1$ for $i = 0,1,\ldots,n-1$.
Therefore, by the previous paragraph, $\theta(V) = \theta(U)$.
But
\begin{eqnarray*}
\theta(U)   & = & \langle \theta(u) \,|\, u \in U^{\#} \rangle \\[.5ex]
& = & \langle C_{\theta(E)}(u) \,|\, u \in U^{\#} \rangle \\[.5ex]
& =& \theta(E),
\end{eqnarray*}
by Theorem~\ref{Kgrp_gen}. Therefore $\theta(V) = \theta(E)$
for any four-subgroup $V < S$.

If now $P \leqslant  S$ is a subgroup of $2$-rank at least $2$, $g \in N_G(P)$ and
$V \leqslant  P$ is a four-subgroup, we have
$$
\theta(E)^g = \theta(V)^g = \theta(V^g) = \theta(E)
$$
and $N_G(P) \leqslant  N_G(\theta(E))$. Therefore $\Gamma_{S,2}(G) \leqslant
N_G(\theta(E))$. \hfill $\Box$

\end{document}